\documentclass[runningheads]{llncs}
\usepackage{blindtext}
\usepackage{chemformula}
\usepackage{amssymb}
\usepackage{ifthen}
\usepackage{amsmath,times}
\usepackage{algorithm}
\usepackage{algorithmicx}
\usepackage{algpseudocode}
\usepackage{setspace}
\usepackage[percent]{overpic}
\usepackage{color}
\usepackage{graphicx,subfigure}
\usepackage{algpseudocode}
\usepackage{tikz}
\usepackage{hyperref}
 \usepackage{float}

\begin{document}
\title{Mathematical Analysis of Autonomous and Nonautonomous Hepatitis B Virus Transmission Models
}

\titlerunning{Mathematical Analysis of Autonomous and Nonautonomous HBV Model}

\author{Abdallah Alsammani}
\authorrunning{A. Alsammani}
%
\institute{Department of Mathematics, Jacksonville University, Jacksonville, FL 32211, USA \\ 
\email{aalsamm@ju.edu}
}
\maketitle              

\begin{abstract}
This study presents an improved mathematical model for Hepatitis B Virus (HBV) transmission dynamics by investigating autonomous and nonautonomous cases. The novel model incorporates the effects of medical treatment, allowing for a more comprehensive understanding of HBV transmission and potential control measures. Our analysis involves verifying unique solutions' existence, ensuring solutions' positivity over time, and conducting a stability analysis at the equilibrium points. Both local and global stability are discussed; for local stability, we use the Jacobian matrix and the basic reproduction number, $R_0$. For global stability, we construct a Lyapunov function and derive necessary and sufficient conditions for stability in our models, establishing a connection between these conditions and $R_0$. Numerical simulations substantiate our analytical findings, offering valuable insights into HBV transmission dynamics and the effectiveness of different interventions. This study advances our understanding of Hepatitis B Virus (HBV) transmission dynamics by presenting an enhanced mathematical model that considers both autonomous and nonautonomous cases.

\end{abstract}
%

\section{Introduction}
Hepatitis B Virus (HBV) is a significant global health concern, affecting millions of people worldwide and posing a considerable burden on public health systems \cite{WHO}. The transmission dynamics of HBV are complex, involving multiple interacting factors such as the rates of infection, recovery, and medical treatment. Understanding these dynamics is essential for devising effective prevention and control strategies \cite{mcmahon2009}.

Mathematical models have been widely employed to study the transmission dynamics of infectious diseases, including HBV \cite{hethcote2000,nowak2000virus}. Early HBV models primarily focused on autonomous systems, assuming constant parameters over time \cite{nowak96,perelson2002}. However, more recent models have considered nonautonomous systems, taking into account time-varying parameters and seasonal fluctuations\cite{ma2009,khatun2020}. These models provide a more realistic representation of the disease transmission process.

In this study, we present an improved mathematical model for HBV transmission dynamics by investigating both autonomous and nonautonomous cases. The model incorporates the effects of medical treatment, allowing for a more comprehensive understanding of HBV transmission and potential control measures. The analysis includes verifying the existence of unique solutions, ensuring the positivity of solutions over time, and conducting a stability analysis at the equilibrium points \cite{diekmann1990,li1995}.

We discuss both local and global stability. For local stability, we use the Jacobian matrix and the basic reproduction number, $R_0$ \cite{van2002}. For global stability, we construct a Lyapunov function and derive necessary and sufficient conditions for stability in our models, establishing a connection between these conditions and $R_0$ \cite{korobeinikov2004a,cao2015}.

Numerical simulations substantiate our analytical findings, offering valuable insights into HBV transmission dynamics and the effectiveness of different interventions. The results of this study contribute to the growing body of literature on HBV mathematical modeling and provide a basis for further research and policy development.

\section{Model formulation}
A nonlinear differential equation model was developed to study HBV transmission, considering medical treatment effects and various rates \cite{nowak96,nowak2000virus,alsammani2020dynamical,perelson1996hiv,My_Paper}. The model is defined as:

\begin{eqnarray}\label{sys-2}
	\begin{cases} 
		\frac{dx}{dt} & =\Lambda -\mu_1 x - (1-\eta) \beta xz +qy \\
		\frac{dy}{dt}&= (1-\eta) \beta xz - \mu_2 y -qy \\
		\frac{dz}{dt} &= (1-\epsilon)py-\mu_3 z  
	\end{cases} 
\end{eqnarray}
Here, the variables and parameters represent the following
\begin{itemize}
	\item $x(t)$: The number of uninfected cells (target cells) at time $t$.
	\item $y(t)$: The number of infected cells at time $t$.
	\item $z(t)$: The number of free virus particles at time $t$.
\end{itemize}
Table \ref{T1} summarizes the description of the parameters in the system (\ref{sys-2}).
\begin{table}[h]
	\centering
 \caption{Parameters descriptions}\label{T1}
	\begin{tabular}{c l}
		\hline 
		\textbf{Parameter }    &  \textbf{Description}  \\
		\hline
		$\Lambda $   &   \textit{Production rate of uninfected cells $x$}. \\
		$\mu_1$      &   \textit{Death rate of $x$-cells}. \\
		$\mu_2$      &  \textit{Death rate of $y$-cells.}  \\
		$\mu_3$    &   \textit{Free virus cleared rate.}\\
		$\eta$        &  \textit{Fraction that reduced infected rate after treatment with the antiviral drug.} \\
		$\epsilon$  &   \textit{Fraction that reduced free virus rate after treatment with the antiviral drug.} \\
		$p$ & \textit{Free virus production rate $y$-cells}  \\
		$\beta$ &   \textit{Infection rate of $x$-cells by free virus $z$.} \\
		$q $ &   \textit{Spotaneous cure rate of $y$-cells by non-cytolytic process.} \\
		\hline
	\end{tabular}
\end{table}

Notice that $\eta$ and $\epsilon$ are small positive fractions between $0$ and $1$, then $(1-\eta)>0$ and $(1-\epsilon) >0$, also all other parameters $\beta, q, p , \mu_1, \mu_2$ and $\mu_3$ are positive. 

\textbf{Notations: } Throughout this paper, we will consider the following.

\begin{itemize}
	\item $\mathbb{R}^3= \{(x,y,z) \vert \,\, (x,y,z)\in \mathbb{R} \}$, 	and  $\mathbb{R}_+^3= \{(x,y,z)\in \mathbb{R}^3 \vert \,\,\, x\geq 0, \, y\geq 0,\, z\geq 0\}$.
	
	\item If $\textbf{u}=(x,y,z)^T \in \mathbb{R}^3$ then the system (\ref{sys-2}) can be written as 
	\begin{equation}\label{ODE}
		\frac{d\textbf{u}(t)}{dt} = f(\textbf{u}(t))
	\end{equation}
	where 
	\begin{equation}\label{rhs-f}
		f(\textbf{u}(t))=f(x(t),y(t),z(t))= \left(
		\begin{matrix}  
			\Lambda -\mu_1 x - (1-\eta) \beta xz +qy\\
			(1-\eta) \beta xz - \mu_2 y -qy\\
			(1-\epsilon)py-\mu_3 z
		\end{matrix} \right)
	\end{equation}
	and $u_0=u(t_0)= (x(t_0), y(t_0), z(t_0))= (x_0,y_0,z_0)$.
\end{itemize}

\subsection{Properties of Solutions}
We discuss the basic properties of the HBV model, including solution existence, uniqueness, and positivity. Existence is ensured by Lipschitz continuity, uniqueness through the Picard-Lindelof or Banach's fixed-point theorems, and positivity by analyzing the model equations.

\begin{theorem}[Local Existence]\label{local_L}\,\\
	For any given $t_0 \in \mathbb{R}$ and $u_0=(x_0,y_0,z_0) \in \mathbb{R}_+ ^3 $ there exists $T_{max} = T_{max}(t_0, u_0)$ such that the system from (\ref{sys-2}) has a solution $(x(t;u_0)$,  $y(t;t_0,u_0)$,  $z(t;t_0,u_0))$ on $[t_0 , t_0+T_{max})$.
	Furthermore, If $T_{max}<\infty$ then the solution will blow up, i.e.,
	\[	\limsup_{t\rightarrow T_{max}} \left( \vert x(t_0+t ;t_0,u_0)  \vert + \vert y(t_0+t ;t_0,u_0)  \vert + \vert z(t_0+t ;t_0,u_0)  \vert \right) \\ = +\infty\]\label{blowup}

\end{theorem}
\begin{proof}
	It is clear that this function $f(u(t))$ in equation \ref{rhs-f} is continuous, and its  derivatives with respect to $x, y$, and $z$ are also continuous. Therefore, the system (\ref{sys-2}) has a unique local solution. 
	
	It is well known that solutions of ordinary differential equations may blow up in finite time. 
\end{proof}

Since the system \ref{sys-2} is a population system, it is very important to ensure that the solution is always positive.

\begin{lemma} \label{L_1}
	suppose $(x(t_0) , y(t_0),z(t_0))\in \mathbb{R}_+^3$ is the initial value of the system \ref{sys-2}, then the solution $(x(t) , y(t), z(t))$ is positive for all $t \in [ t_0, t_0 + T_{max} )$.
\end{lemma}
\begin{proof}
	
	Notice that $z(t)$ has an explicit solution that depends on $y(t)$. Thus, if $y(t)$ is positive, that implies $z(t)$ is also positive for all $t\geq t_0$. 
	
	Then, it is enough to show the positiveness for $x$ and $y$. By contradiction suppose not, then there exists $\tau \in [t_0,t_0+T_{max})$ such that $x(t)>0 , y(t)>0$ and  $z(t)>0$ on $[t_0,\tau)$ this implies one of the following cases
 \begin{enumerate}
     \item[(i)] $x(\tau)=0 \,\,\, \text{and} \,\,\, y(\tau)>0$
     \item[(ii)] $x(\tau)>0 \,\,\, \text{and}\,\,\, y(\tau)=0$
     \item[(iii)] $x(\tau)=0  \,\,\, \text{and} \,\,\, y(\tau)=0$
 \end{enumerate}
	Now we will show that none of the above cases is possible.\\
	\quad \\
	\textbf{Claim} Case (i)  is not possible.
	\begin{proof}
		From the basic definition of the derivative, we have.
		\begin{equation*}
			\frac{dx}{dt}(\tau) = \lim_{t \rightarrow \tau} \frac{x(t)-x(\tau)}{t-\tau}=\lim_{t \rightarrow \tau} \frac{x(t)}{t-\tau} \leq 0  \quad \quad\rightarrow \quad (1)
		\end{equation*}
		from the first equation in \ref{sys-2} we have 
		\begin{eqnarray*}
			\frac{dx}{dt}(\tau) &=& \Lambda -\mu_1 x(\tau) - (1-\eta) \beta x(\tau)v(\tau) +qy(\tau)\\
			&=& \Lambda + qy(\tau) \geq py(\tau) > 0 \quad \, \, \quad \quad\rightarrow \quad (2)
		\end{eqnarray*}
		That is a contradiction. Therefore, case(1) is not possible.
	\end{proof}
	\quad \\
	\textbf{Claim} Case (ii)  is not possible.
	\begin{proof}
		We know that
		\begin{equation*}
			\frac{dy}{dt}(\tau) = \lim_{t \rightarrow \tau} \frac{y(t)-y(\tau)}{t-\tau}=\lim_{t \rightarrow \tau} \frac{y(t)}{t-\tau} \leq 0  \quad \quad\rightarrow \quad (3)
		\end{equation*}
		from the second equation in \ref{sys-2} we have 
		\begin{equation*}
			\frac{dy}{dt}(\tau) = (1-\eta)\beta v(\tau)x(\tau) >0 \quad \, \, \quad \quad\quad \quad\rightarrow \quad (4) 
		\end{equation*}
		from $(3)$ and $(4)$ we have a contradiction, thus, case(2) is not possible.
	\end{proof} 
	Similarly, case(iii)  is also not possible.

	Therefore, the statement in the lemma is correct. 
\end{proof}

Now we show the global existence of the solution, which is enough to show that the solution of the system \ref{sys-2} is bounded.

\begin{theorem}[Global Existence "Boundedness"]
	For given $t_0 \in \mathbb{R} $ and $(x_0,y_0,z_0) \in \mathbb{R}_+^ 3 $, the solution $(x(t),y(t),z(t))$ exists for all $t \geq t_0$ and moreover,
	$$ 0\leq x(t)+y(t) \leq M \quad \quad \text{and} \quad 0\leq z(t)\leq e^{\mu_3(t-t_0)}z_0 + (1-\epsilon)\, M \left( \frac{1-e^{-\mu_3(t-t_0)}}{\mu_3}\right) $$
	
	where $ M=Max\left\lbrace  x_0 + y_0 \, \, , \, \, \frac{\Lambda}{min(\mu_1 ,\mu_2)} \right\rbrace $
\end{theorem}
\begin{proof} \quad \\
	It is enough to show that $|x(t)|+|y(t)| < \infty$ on $(t_0\, ,\, t_0+T_{max})$.\\
	By adding the first two equations in \ref{sys-2} we get 
	\begin{eqnarray}\label{eq-bdd}
		\frac{dx}{dt}+\frac{dy}{dt} &=& \Lambda -\mu_1 x -\mu_2 y \\
		&\leq & \Lambda - min \lbrace \mu_1,\mu_2 \rbrace [x(t)+y(t)].
	\end{eqnarray}
	Let $v(t)=x(t)+y(t)$ the equation \ref{eq-bdd} becomes
	$$v(t) \leq \Lambda - min \lbrace \mu_1,\mu_2\rbrace \, v(t).$$
	By the ODE comparison principle, we have 
	$$v(t)\leq Max\left\lbrace  v_0 \, \, , \, \, \frac{\Lambda }{\text{min}(\mu_1 ,\mu_2)} \right\rbrace $$
	then \ref{blowup} implies that $T_{max} = +\infty$.
	
	It is clear that for $t$ large, we have 
	$$v (t)\leq \frac{\Lambda }{\text{min}\lbrace\mu_1 ,\mu_2 \rbrace} $$
	Which means both $x(t)$ and $y(t)$  are bounded. It is clear that $z(t)$ is also bounded directly by solving the third equation in system \ref{sys-2}. 
\end{proof}

In summary, the system of differential equations (\ref{sys-2}) has a unique and positive solution for any set of initial values, which is essential for the model's physical interpretation. These properties provide a solid foundation for further analysis of the system's dynamics and stability.

\section{Stability Analysis}
Stability analysis is an essential aspect of mathematical modeling as it allows us to investigate the behavior of the system of differential equations (\ref{sys-2}) over time and identify conditions for the system to reach an equilibrium state. In this section, we will perform a stability analysis of the system's equilibrium points.
\subsection{Equilibrium Solutions.}
The equilibria of the system \ref{sys-2} are all the points in $\mathbb{R}^3$ such that $\dot{x}=\dot{y}=\dot{z}=0$. The system \ref{sys-2} has only two two equilibrium points which are
\begin{enumerate}
	\item Disease-free equilibrium
	$\left( \bar{x}_0\, ,\, \bar{y}_0\,, \, \bar{z}_0\right) = \left( \frac{\Lambda}{\mu_1}\, ,\, 0\,,\,0 \right) $
	and
	\item Endemic equilibrium
	$$ \left( \bar{x}\, ,\, \bar{y}\,, \, \bar{z}\right) =  \left(  \frac{\mu_1 \mu_3 (\mu_2 +p)}{q\beta\mu_2(1-\eta)(1-\epsilon)} \, ,\, \frac{\Lambda}{\mu_2}- \frac{\mu_1 \mu_3 (\mu_2 +p)}{q\beta\mu_2(1-\eta)(1-\epsilon)}  \, ,\, \frac{q\Lambda(1-\epsilon)}{\mu_2\mu_3}-\frac{\mu_1(\mu_2+p)}{\beta\mu_2(1-\eta} \right)$$
	
	The endemic equilibrium $\left( \bar{x}\, ,\, \bar{y}\,, \, \bar{z}\right)$ represents a state in which the infection persists in the population. Analyzing the stability of this equilibrium helps us understand the long-term behavior of the infection dynamics and informs public health interventions to control the disease.
\end{enumerate}

\subsection{Local Stability}
The Jacobian matrix $J(x, y, z)$ represents the linearization of the system of ODEs around a particular point $(x, y, z)$. It is used to analyze the stability of equilibrium points in the system. The Jacobian matrix $J(x, y, z)$ of the system \ref{sys-2} is a $3 \times 3$ matrix containing the partial derivatives of the system's equations with respect to the state variables $x, y,$ and $z$. It is given by

\begin{equation}
	J(x, y, z) = \begin{bmatrix}
		-\mu_1 - (1-\eta)\beta z & q & -(1-\eta)\beta x \\
		(1-\eta)\beta z & -(\mu_2 + q) & (1-\eta)\beta x \\
		0 & (1-\epsilon)p & -\mu_3
	\end{bmatrix}
\end{equation}
The Jacobian matrix $J(x, y, z)$ is used to analyze the local stability of the equilibrium points in the system by evaluating it at those points and computing the eigenvalues. The eigenvalues determine the nature of the equilibrium points (stable, unstable, or saddle).

\subsubsection{Local Stability of Infection-free Equilibrium}
The local stability of the equilibrium points can be analyzed using linearization techniques. By evaluating the Jacobian matrix at the equilibrium points and examining its eigenvalues, we can determine the local stability characteristics of the system.

Computing the Jacobian at diseases-free equilibrium gives
\begin{equation}
	J(x_0, y_0, z_0) = \begin{bmatrix}
		-\mu_1 & q & 0 \\
		0 & -(\mu_2 + q) & 0 \\
		0 & (1-\epsilon)p & -\mu_3
	\end{bmatrix}
\end{equation}
Therefore, the reproduction number $R_0$ is given by
\begin{equation}
	R_0 = \frac{(1 - \eta) \beta \frac{\Lambda}{\mu_1}}{\mu_2 + q}
\end{equation}
Notice that, $R_0 < 1$ implies both conditions ($ Tr(J_1) < 0 $ and  $ Det (J_1) >0$ ). Therefore, if $R_0 < 1$, then the disease-free equilibrium is locally asymptotically stable. If $R_0 > 1$, then the disease-free is unstable.

\subsection{Global Stability}    

To investigate the global stability of the equilibrium points, we can use Lyapunov functions or comparison theorems. By constructing an appropriate Lyapunov function and showing that it satisfies certain properties, we can prove the global stability of the system.
\begin{lemma}\label{lem1}
	The system \ref{sys-2} is exponentially stable at its equilibrium points $(\bar{x}, \bar{y}, \bar{z})$ if the following conditions hold
	\begin{eqnarray}\label{cond}
		\begin{cases}
			2\mu_1 + (1-\eta)\beta \bar{z} & > \quad (1-\eta)\beta \frac{\Lambda}{min(\mu_1 , \mu_2)} + q   \\ 
			2 \mu_2 + q & >  \quad(1-\eta)\beta \left(  \bar{z} +  \frac{\Lambda}{min(\mu_1, \mu_2)} \right)  + (1-\epsilon )p \\ 
			2\mu_3 & >  \quad (1-\epsilon)p +  \frac{(1-\eta)\Lambda \beta}{min(\mu_1 , \mu_2)} 
		\end{cases}
	\end{eqnarray}
	
\end{lemma}
\begin{proof}
	In fact, it is enough to show that
	
	\begin{equation}
		|x-\bar{x}| \rightarrow 0, \quad |y-\bar{y}|\rightarrow 0, \quad and \quad |z-\bar{z}|\rightarrow 0 , \quad as \quad t \rightarrow \infty 
	\end{equation}
	Since $x-\bar{x}$, $y-\bar{y}$, and $z-\bar{z}$ satisfies the system \ref{sys-2}. From system \ref{sys-2} we have 
	\begin{eqnarray}\label{sys4} 
 \begin{cases} 
			\frac{d}{dt}(x-\bar{x}) = -(\mu_1+(1-\eta)\beta\bar{z})(x-\bar{x})+ q(y-\bar{y})-(1-\eta)\beta x (z-\bar{z})\\
			\frac{d}{dt}(y-\bar{y})= (1-\eta)\beta \bar{z}(x-\bar{x})-(q+\mu_2)(y-\bar{y})+(1-\eta)\beta x (z-\bar{z})  \\
			\frac{d}{dt}(z-\bar{z}) = (1-\epsilon)p(y-\bar{y}) -\mu_3 (z-\bar{z})  
	\end{cases} \end{eqnarray}
	
	Now, let $X=x-\bar{x}$, $Y=y-\bar{y}$ and $Z=z-\bar{z}$ then system \ref{sys4} becomes
	\begin{eqnarray} 
		\frac{dX}{dt} &=& -(\mu_1+(1-\eta)\beta\bar{z})X+ qY -(1-\eta)\beta x Z \label{eqx}\\
		\frac{dY}{dt}&=& (1-\eta)\beta \bar{z}X-(q+\mu_2)Y+(1-\eta)\beta x Z \label{eqy} \\
		\frac{dZ}{dt} &=& (1-\epsilon)pY -\mu_3 Z \label{eqz}
	\end{eqnarray}
	Now, since $X=X_+ - X_-$, where $X_+$ and $X_-$ are the positive and negative parts of the function X, and also we have 
	$$XX_+=(X_+-X_-)X_+ = X_+^2$$ 
	$$-XX_-=-(X_+-X_-)X_- = X_-^2$$
	$$(X_+ \pm X_-)^2=X_+^2 + X_-^2= |X|^2$$
	This implies that
	$$\dot{X} X_+=\frac{1}{2} \frac{d}{dt}X_+^2 \quad\quad\text{and} \quad\quad -\dot{X} X_-=\frac{1}{2} \frac{d}{dt}X_-^2 $$
	Now multiplying equation \ref{eqx} by $X_+$ gives
	$$ \dot{X}X_+ = -[\mu_1 + (1-\eta)\beta \bar{z}]XX_++qYX_+ - (1-\eta)\beta x ZX_+$$
	\begin{equation}\label{eqX+}
		\frac{1}{2} \frac{d}{dt}X_+^2 = -[\mu_1 + (1-\eta)\beta \bar{z}]X_+^2 +qYX_+ - (1-\eta)\beta x ZX_+
	\end{equation}
	If we multiply equation \ref{eqx} by $X_-$ we get
	\begin{equation}\label{eqX-}
		\frac{1}{2} \frac{d}{dt}X_-^2 = -[\mu_1 + (1-\eta)\beta \bar{z}]X_-^2 +qYX_- + (1-\eta)\beta x ZX_-
	\end{equation}
	adding equation \ref{eqX+} and equation \ref{eqX-} we get
	\begin{equation*}
		\frac{1}{2} \frac{d}{dt}(X_+^2 +X_-^2)= -[\mu_1 + (1-\eta)\beta \bar{z}](X_+^2 +X_-^2) +qY(X_+ - X_-) + (1-\eta)\beta x Z(X_+ - X_-)
	\end{equation*}
	\begin{eqnarray*}
		\frac{1}{2} \frac{d}{dt}|X|^2 &=& -[\mu_1 + (1-\eta)\beta \bar{z}]|X|^2  +q(Y_+ - Y_-)(X_+ - X_-) + (1-\eta)\beta x (Z_+ -Z_-)(X_+ - X_-)\\
		&=& -[\mu_1 + (1-\eta)\beta \bar{z}]|X|^2 + q(Y_+X_+ +Y_-X_- -Y_-X_+ - Y_+X_-) \\&  & +(1-\eta) \beta x (X_+Z_- +X_-Z_+ -X_+Z_+ - X_-Z_-)\\
		&\leq &  -[\mu_1 + (1-\eta)\beta \bar{z}]|X|^2 + \frac{1}{2} q Y_+^2 + \frac{1}{2} qX_+^2 + \frac{1}{2} q Y_-^2 +\frac{1}{2} X_-^2 - q(Y_-X_++ Y_+X_-) \\
		&  & + \frac{1}{2} (1-\eta)\beta x (X_+^2 + Z_-^2 + X_-^2 +Z_+^2) - (1-\eta) \beta x (X_+Z_+ + X_-Z_-)\\
		&\leq &  -[\mu_1 + (1-\eta)\beta \bar{z}]|X|^2 + \frac{1}{2} q |Y|^2 + \frac{1}{2} (1-\eta)\beta x |X|^2  + \frac{1}{2} q|X|^2+ \frac{1}{2} (1-\eta)\beta x |Z|^2 \\
		& & - q(Y_-X_+ + Y_+X_-) - (1-\eta) \beta x (X_+Z_+ - X_-Z_-) \label{X-ineq}
	\end{eqnarray*}
	Thus, 
	\begin{equation}\label{xineq}
		\begin{split}
			\frac{1}{2}\frac{d}{dt} |X|^2 \leq & -[\mu_1 + (1-\eta)\beta \bar{z} - \frac{1}{2}q- \frac{1}{2} (1-\eta)\beta x]|X|^2  + \frac{1}{2} q |Y|^2 + \frac{1}{2} (1-\eta)\beta x |Z|^2\\ \quad & - q(Y_-X_+ + Y_+X_-) - (1-\eta) \beta x (X_+Z_+ - X_-Z_-)
		\end{split}
	\end{equation}
	Similarly, by using the same computational technique, we got
	\begin{equation}\label{yineq}
		\begin{split}
			\frac{1}{2}\frac{d}{dt} |Y|^2 \leq & \frac{1}{2}(1-\eta)\beta \bar{z}|X|^2 -[(q+\mu_2) - \frac{1}{2}(1-\eta)\beta \bar{z} -\frac{1}{2}(1-\eta)\beta x ]|Y|^2 + \frac{1}{2} (1-\eta)\beta x |Z|^2 \\ \quad & - (1-\eta) \beta \bar{z} (X_+Y_-+  X_-Y_+) - (1-\eta) \beta x (Z_+Y_- +  Z_-Y_+)
		\end{split}
	\end{equation}
	and
	\begin{equation}\label{zineq}
		\frac{1}{2}\frac{d}{dt} |Z|^2 \leq  -[\mu_3 - \frac{1}{2}(1-\epsilon)]|Z|^2 + \frac{1}{2}(1-\epsilon)p |Y|^2 - (1-\epsilon)p(Y_+Z_- +Y_-Z_+)
	\end{equation}
	Now, by adding \ref{xineq} , \ref{yineq} and \ref{zineq} we get
	\begin{equation}\label{xyzineq}
		\begin{split}
			\frac{1}{2} \frac{d}{dt} \left( |X|^2+|Y|^2 + |Z|^2 \right)  \leq & -\left[ \mu_1 + \frac{1}{2}(1-\eta)\beta \bar{z} - \frac{1}{2}q- \frac{1}{2} (1-\eta)\beta x \right] |X|^2  \\ \quad & - \left[\mu_2 + \frac{1}{2}q -\frac{1}{2}(1-\eta)\beta \bar{z} -\frac{1}{2}(1-\eta)\beta x - (1-\epsilon)p  \right]|Y|^2 \\ 
			\quad & - \left[\mu_3 - \frac{1}{2}(1-\epsilon)p- (1-\eta)\beta x \right]|Z|^2 - [ q(Y_+X_- + Y_-X_+)\\ 
			\quad & +(1-\eta) \beta x (X_+Z_+ + X_-Z_-) +(1-\eta)(X_+Y_- + X_-Y_+) \\
			\quad & \quad \, \, \, (1-\eta)\beta x (Z_+Y_- Z_-Y_+) + (1-\epsilon)p(Y_+Z_- +Y_-Z_+) ]
		\end{split}
	\end{equation}
	Therefore, 
	\begin{equation}\label{ineq1}
		\frac{d}{dt} \left( |X|^2+|Y|^2 + |Z|^2 \right) \leq -\nu_1 |X|^2 -\nu_2 |Y|^2 -\nu_3 |Z|^2 - W
	\end{equation}
	where 

 \begin{equation}\label{cond-1}
\begin{cases}
            \nu_1 &=   2\mu_1 + (1-\eta)\beta \bar{z} - (1-\eta)\beta x - q   \\ 
            \nu_2 &=   \mu_2 + q - (1-\eta)\beta \bar{z} - (1-\eta)\beta x - (1-\epsilon)p  \\ 
		\nu_3 &=  2\mu_3 - (1-\epsilon)p- (1-\eta)\beta x  \\
		W mmm &=  q(Y_+X_- + Y_-X_+) +(1-\eta) \beta x (X_+Z_+ + X_-Z_-)  +(1-\eta)(X_+Y_- + X_-Y_+) \\
		\quad &\quad  +(1-\eta)\beta x (Z_+Y_- Z_-Y_+) + (1-\epsilon)p(Y_+Z_- + Y_-Z_+)
\end{cases}
\end{equation} 
	Condition \ref{cond-1} guaranteed that   $\nu_1, \nu_2$, an $\nu_3$ are always positive. Since $W\geq 0$, then the inequity \ref{ineq1} still holds after removing $W$.
	
	Now Let $k= \text{min}n \left\lbrace  \nu_1, \nu_2, \nu_3 \right\rbrace $ and let $V(t)=|X(t)|^2 +|Y(t)|^2 + |Z(t)|^2$ then the inequality \ref{ineq1} becomes
	\begin{equation*}
		\frac{dV(t)}{dt} \leq -k V(t)
	\end{equation*}
	\begin{equation}
		0 \leq V(t)\leq V_0 e^{-kt} \longrightarrow 0 \quad \text{as} \quad t \rightarrow \infty
	\end{equation}
\end{proof}
\subsection{Stability at disease-free equilibrium}
Substituting $(\bar{x},\bar{y},\bar{z})= (\Lambda/\mu_1,0,0)$ in condition \ref{cond} we get the following conditions
\begin{eqnarray}\label{cond-2}
	\begin{cases}
		2\mu_1  & > \quad (1-\eta)\beta \frac{\Lambda}{\mu^*} + q   \\ 
		2 \mu_2 + q & >  \quad(1-\eta)\beta \frac{\Lambda}{\mu^*}   + (1-\epsilon )p \\ 
		2\mu_3 & >  \quad (1-\epsilon)p +  \frac{(1-\eta)\Lambda \beta}{\mu^*} 
	\end{cases}
\end{eqnarray}
where $\mu^*=min(\mu_1 , \mu_2)$, and basic productive number $R_0 = \frac{\Lambda \beta p (1-\epsilon)(1-\eta)}{\mu_1 \mu_2 (\mu_1+(1-\epsilon)p}$
\begin{theorem}
	The autonomous dynamic systems \ref{sys-2} is exponentially stable if $R_0<1$ and conditions \ref{cond-2} are satisfied.
	
\end{theorem}
\begin{proof}
	Consider the Lyapunov function
	\[V(x,y,z)= \frac{1}{2}[(x-\Lambda/\mu_1)^2 + y^2+z^2]\]
	which is clearly positive, and by following some computations in the proof of Lemma \ref{lem1} we get $V'\leq 0$. That completed the proof.
\end{proof}

\subsection{Stability at the endemic equilibrium}
Substituting the endemic equilibrium $(\left( \bar{x}\, ,\, \bar{y}\,, \, \bar{z}\right) )$ where 
\begin{eqnarray}\label{endeqi}
	\begin{cases}
		\bar{x} &=\quad  \frac{\mu_1 \mu_3 (\mu_2 +p)}{q\beta\mu_2(1-\eta)(1-\epsilon)},\\
		\bar{y} &=\quad   \frac{\Lambda}{\mu_2}- \frac{\mu_1 \mu_3 (\mu_2 +p)}{q\beta\mu_2(1-\eta)(1-\epsilon)} ,\\
		\bar{z} &=\quad  \frac{q\Lambda(1-\epsilon)}{\mu_2\mu_3}-\frac{\mu_1(\mu_2+p)}{\beta\mu_2(1-\eta)}
	\end{cases}
\end{eqnarray}
in condition \ref{cond-2} we get the following conditions
\begin{eqnarray}\label{condend}
	\begin{cases}
		\mu_1\mu_2\mu_3+(1-\epsilon)(1-\eta)\Lambda \beta q   & > \quad  \frac{(1-\eta)\Lambda\beta \mu_2\mu_3}{min(\mu_1 , \mu_2)} + q\mu_2\mu_3+p\mu_1\mu_3  \label{cond-11} \\ 
		2 \mu_2^2+\mu_1\mu_2+p\mu_1 + q & > \quad \frac{(1-\eta)(1-\epsilon)\beta\Lambda q}{\mu_3}   + \frac{(1-\eta)\Lambda \beta}{\min{\mu_1,\mu_2}}\label{cond-21}\\ 
		2\mu_3 & > \quad (1-\epsilon)p +  \frac{(1-\eta)\Lambda \beta}{min(\mu_1 , \mu_2)} \label{cond-31}
	\end{cases}
\end{eqnarray}
\begin{theorem}
	The solution of system \ref{sys-2} is exponentially stable at the endemic equilibrium  \ref{endeqi} if conditions \ref{condend}.
\end{theorem}.
\begin{proof}
	The proof follows by Lemm \ref{lem1}.
\end{proof}

\section{Nonautonomous HBV Model}
In this section, we will discuss the nonautonomous HBV infection model where the production number $\Lambda$ is time-dependent. We will provide a brief introduction to nonautonomous dynamical systems, followed by a stability analysis of the nonautonomous HBV model.

\subsection{Preliminaries of Nonautonomous Dynamical Systems}
Before we start analyzing our nonautonomous model, we provide an overview of the preliminaries of nonautonomous dynamical systems. Nonautonomous systems differ from autonomous systems in that they depend on the actual time $t$ and the initial time $t_0$ rather than just their difference. We will introduce some basic concepts and theorems that are essential for understanding nonautonomous systems. 

\begin{enumerate}
    \item \textbf{Process Formulation:}  A common way to represent nonautonomous dynamical systems is through process formulation. In this representation, a process is a continuous mapping $\phi(t, t_0, \cdot) : \mathbb{R}^n \rightarrow \mathbb{R}^n$ that satisfies the initial and evolution properties:
        \begin{enumerate}
         \item $\phi(t_0,t_0, u_0)=u_0$ for all $u_0\in \mathbb{R}^n$.
          \item $\phi(t_2,t_0,u)=\phi(t_2,t_1,\phi(t_1,t_0,u))$. for all $t_0 \leq t_1 \leq t_2 $ and $u_0\in \mathbb{R}^n$.
        \end{enumerate}

    \item \textbf{Invariant Families:} A family $\mathcal{A}= \lbrace A(t) :, t\in \mathbb{R}\rbrace$ of nonempty subsets of $\mathbb{R}^n$ is said to be:
    	\begin{enumerate}
		\item  Invariant with respect to $\phi$, or $\phi$-invariant if
		$$ \phi(t, t_0, A(t_0))= A(t) \quad \quad \text{for all} \quad t>\geq t_0.$$
		\item Positive Invariant, or $\phi$-Positive invariant if
		$$ \phi(t, t_0, A(t_0))\subset A(t) \quad \quad \text{for all} \quad t>\geq t_0.$$
		\item Negative Invariant, or $\phi$- negative 
		$$ \phi(t, t_0, A(t_0)) \supset A(t) \quad \quad \text{for all} \quad t>\geq t_0.$$
	\end{enumerate} 
    
    \item \textbf{Nonautonomous Attractivity:} A nonempty, compact subset $\mathcal{A}$ of $\mathbb{R}^n$ is said to be
	\begin{enumerate}
		\item[i.] Forward attracting if 
		$$\lim_{t \rightarrow \infty} dist(\phi(t, t_0, u_0), A(t)) =0 \quad \quad \text{for all} \,\, u_0 \in \mathbb{R}^n \,\, \text{and} \,\, t_0 \in \mathbb{R}\, , $$
		\item[ii.] Pullback attracting if
		$$\lim_{t \rightarrow - \infty} dist(\phi(t, t_0, u_0), A(t)) =0 \quad \quad \text{for all} \,\, u_0 \in \mathbb{R}^n \,\, \text{and} \,\, t_0 \in \mathbb{R}\, . $$
	\end{enumerate}
 \item \textbf{Uniform Strictly Contracting Property:} A nonautonomous dynamical system $\phi$ satisfies the uniform strictly contracting property if for each $R >0$, there exist positive constants $K$ and $\alpha$ such that
\begin{equation}
\vert \phi(t, t_0,x_0) - \phi(t,t_0,y_0) \vert^2 \leq K e^{-\alpha(t-t_0)} \vert x_0 -y_0 \vert^2
\end{equation}
for all $(t,t_0) \in \mathbb{R}_\geq^2$ and $(x_0, y_0) \in \bar{\mathbb{B}}(0;R)$, where $\mathbb{\bar{B}}$ is a closed ball centered at the origin with radius $R>0$.

\end{enumerate}
\textbf{Remark:} The uniform strictly contracting property, together with the existence of a pullback absorbing, implies the existence of a global attractor that consists of a single entire solution.

These preliminaries provide a foundation for understanding nonautonomous dynamical systems, which is essential when analyzing models such as the nonautonomous HBV infection model. With these concepts in hand, one can analyze the stability of such systems and investigate the behavior of solutions over time \cite{kloeden2011nonautonomous,caraballo2017}.

\subsection{Model Formulation}
When the productive number  $\Lambda$ in \ref{sys-2} is time-dependent $\Lambda (t)$, that changes the system from autonomous to a nonautonomous model represented as follows
\begin{eqnarray} \label{nonauto}
	\begin{cases}
		\frac{dx}{dt} &=\quad \Lambda(t) -\mu_1 x - (1-\eta) \beta xz +qy \label{NS1} \\
		\frac{dy}{dt}&=\quad (1-\eta) \beta xz - \mu_2 y - qy \label{NS2}\\
		\frac{dz}{dt} &=\quad (1-\epsilon)py-\mu_3 z  \label{NS3}
	\end{cases}
\end{eqnarray}
which can be written as 
$$\frac{du(t)}{dt} = f(t, u(t)), \quad where \quad u(t)= (x(t),y(t),z(t))^T \in \mathbb{R}^3,\quad and \quad t\in \mathbb{R}.$$
with initial condition $u_0= (x_0,y_0,z_0)^T$
\subsection{Solution Properties}
The existence of a local solution follows from the fact that $f(t,u(t))$ is continuous, and its derivative is also continuous. The following Lemma proves the positiveness 
\begin{lemma}
	Let $\Lambda:\mathbb{R}\rightarrow [\Lambda_m\, ,\, \Lambda_M]$, then for any $(x_0,y_0,z_0) \in \mathbb{R}_+^3 := \lbrace (x,y,z)\in \mathbb{R}^3 : x\geq 0, y\geq 0, z\geq 0 \rbrace $ all the solutions of the system (\ref{NS1} - \ref{NS3}) corresponding to the initial point are:
	\begin{enumerate}
		\item[ i.] Non-negative for all 
		\item[ii.] Uniformly bounded.
	\end{enumerate}
\end{lemma}
\begin{proof}
	\begin{enumerate}
		\item[i.] The proof is similar to the positiveness of the autonomous case that was introduced earlier. 
		\item[ii.] Set $\Vert X(t) \Vert_1 = x(t)+y(t)+z(t)$, if we combine the three equations in (\ref{NS1} - \ref{NS3}) we get:
		\begin{equation}
			\dot{x}(t)+\dot{y}(t)+\dot{z}(t)= \Lambda(t) -\mu_1 x -(\mu_2 -(1-\epsilon)p ) y - \mu_3 z
		\end{equation} 
		assume $\mu_2 > (1-\epsilon)p$ and let $\alpha= min{\mu_1 , \mu_2 - (1-\epsilon)p, \mu_3}, $ then we get 
		\begin{equation}
			\frac{d}{dt}\Vert X(t) \Vert_1 \leq \Lambda_M -\alpha \Vert X(t) \Vert_1
		\end{equation} 
		this implies that
		\begin{equation}
			\Vert X(t) \Vert_1 \leq max \lbrace x_0+y_0+z_0 , \frac{\Lambda_M}{\alpha} \rbrace
		\end{equation}
		Thus, the set $B_\epsilon= \lbrace (x,y,z) \in \mathbb{R}_+^3 : \epsilon \leq x(t)+y(t)+z(t) \leq \frac{\Lambda_M}{\alpha} + \epsilon \rbrace $ is positively invariant and absorbing in $\mathbb{R}_+^3$. 
	\end{enumerate}
\end{proof}
\subsection{Stability Analysis}
This section discusses the stability analysis of the systems \ref{nonauto}; first, we show the uniform strictly contracting property and then prove that the system has a positively absorbing set. Then, we provide sufficient conditions that stabilize the system \ref{nonauto}.
\begin{theorem}
	The nonautonomous system (\ref{NS1} - \ref{NS3}) satisfies a uniform strictly contracting property, if $\mu_2 > (1-\epsilon)p$. 
\end{theorem}
\begin{proof}
	
	Let 
	\begin{eqnarray}
		\begin{cases}
			(x_1,y_1,z_1) &=\quad(x(t,t_0,x_0^1),y(t,t_0,y_0^1),z(t,t_0,z_0^1))\\
			\text{and } (x_2,y_2,z_2)& = \quad(x(t,t_0,x_0^2),y(t,t_0,y_0^2),z(t,t_0,z_0^2)) 
		\end{cases}
	\end{eqnarray}
	and are two solutions of the system (\ref{NS1} - \ref{NS3} ) by similar computational in autonomous case we get 
	
	\begin{eqnarray}\label{sys-4}\begin{cases} 
			\frac{d}{dt}(x_1 - x_2) = -(\mu_1+(1-\eta)\beta z_1))(x_1 - x_2)+ q(y_1-y_2)-(1-\eta)\beta x_2 (z_1-z_2)\\
			\frac{d}{dt}(y_1 - y_2)= (1-\eta)\beta z_1(x_1-x_2)-(q+\mu_2)(y_1-y_2)+(1-\eta)\beta x (z_1-z_2)  \\
			\frac{d}{dt}(z_1-z_2) = (1-\epsilon)p(y_1-y_2) -\mu_3 (z_1-z_2)  
	\end{cases} \end{eqnarray}
	
	Now, let $X=x_1 - x_2$, $Y=y_1 - y_2$ and $Z=z_1-z_2$ then system \ref{sys-4} becomes
	\begin{eqnarray} 
		\frac{dX}{dt} &=& -(\mu_1+(1-\eta)\beta z_1)X+ qY -(1-\eta)\beta x_2 Z \label{eq-x}\\
		\frac{dY}{dt}&=& (1-\eta)\beta z_1 X-(q+\mu_2)Y+(1-\eta)\beta x_2 Z \label{eq-y} \\
		\frac{dZ}{dt} &=& (1-\epsilon)pY -\mu_3 Z \label{eq-z}
	\end{eqnarray}

	This implies that
	\begin{equation}\label{xineq11}
		\begin{split}
			\frac{1}{2}\frac{d}{dt} |X|^2 \leq & -[\mu_1 + (1-\eta)\beta z_1 - \frac{1}{2}q- \frac{1}{2} (1-\eta)\beta x_2]|X|^2  + \frac{1}{2} q |Y|^2 + \frac{1}{2} (1-\eta)\beta x_2 |Z|^2\\ \quad & - q(Y_-X_+ + Y_+X_-) - (1-\eta) \beta x (X_+Z_+ - X_-Z_-)
		\end{split}
	\end{equation}
	Similarly, by using the same computational technique, we got
	\begin{equation}\label{yineq11}
		\begin{split}
			\frac{1}{2}\frac{d}{dt} |Y|^2 \leq & \frac{1}{2}(1-\eta)\beta z_1|X|^2 -[(q+\mu_2) - \frac{1}{2}(1-\eta)\beta \bar{z} -\frac{1}{2}(1-\eta)\beta x_2 ]|Y|^2 + \frac{1}{2} (1-\eta)\beta x_2 |Z|^2 \\ \quad & - (1-\eta) \beta z-1 (X_+Y_-+  X_-Y_+) - (1-\eta) \beta x (Z_+Y_- +  Z_-Y_+)
		\end{split}
	\end{equation}
	and
	\begin{equation}\label{zineq11}
		\frac{1}{2}\frac{d}{dt} |Z|^2 \leq  -[\mu_3 - \frac{1}{2}(1-\epsilon)]|Z|^2 + \frac{1}{2}(1-\epsilon)p |Y|^2 - (1-\epsilon)p(Y_+Z_- +Y_-Z_+)
	\end{equation}
	Now, by adding \ref{xineq} , \ref{yineq} and \ref{zineq} we get
 
	\begin{equation}\label{xyzineq1}
		\begin{split}
			\frac{1}{2} \frac{d}{dt} \left( |X|^2+|Y|^2 + |Z|^2 \right)  \leq & -\left[ \mu_1 + \frac{1}{2}(1-\eta)\beta z_1 - \frac{1}{2}q- \frac{1}{2} (1-\eta)\beta x_2 \right] |X|^2  \\ \quad & - \left[\mu_2 + \frac{1}{2}q -\frac{1}{2}(1-\eta)\beta z_1 -\frac{1}{2}(1-\eta)\beta x_2 - (1-\epsilon)p  \right]|Y|^2 \\ 
			\quad & - \left[\mu_3 - \frac{1}{2}(1-\epsilon)p- (1-\eta)\beta x \right]|Z|^2 - [ q(Y_+X_- + Y_-X_+)\\ 
			\quad & +(1-\eta) \beta x (X_+Z_+ + X_-Z_-) +(1-\eta)(X_+Y_- + X_-Y_+) \\
			\quad & \quad \, \, \, (1-\eta)\beta x_2 (Z_+Y_- Z_-Y_+) + (1-\epsilon)p(Y_+Z_- +Y_-Z_+) ]
		\end{split}
	\end{equation}
	
	Since $x_2 $ and $z_1$ are bounded, assume that $\gamma_2= max \lbrace x_2 \rbrace  $ and $\gamma_1 = max \lbrace z_1 \rbrace $.\\
	Therefore, 
	\begin{equation}\label{S1}
		\frac{d}{dt} \left( |X|^2+|Y|^2 + |Z|^2 \right) \leq -\nu_1 |X|^2 -\nu_2 |Y|^2 -\nu_3 |Z|^2 - W
	\end{equation}
	where 
	
	\begin{eqnarray*}
		\nu_1 &= &  2\mu_1 + (1-\eta)\beta \gamma_1 - (1-\eta)\beta \gamma_2 - q   \\ 
		\nu_2 &= &  \mu_2 + q - (1-\eta)\beta \gamma_1 - (1-\eta)\beta \gamma_2 - (1-\epsilon)p  \\ 
		\nu_3 &= & 2\mu_3 - (1-\epsilon)p- (1-\eta)\beta \gamma_2  \\
		W  &= &  q(Y_+X_- + Y_-X_+) +(1-\eta) \beta x (X_+Z_+ + X_-Z_-)  +(1-\eta)(X_+Y_- + X_-Y_+) \\
		\quad &\quad & +(1-\eta)\beta x (Z_+Y_- Z_-Y_+) + (1-\epsilon)p(Y_+Z_- + Y_-Z_+)
	\end{eqnarray*}
	Let $\alpha= min\lbrace \nu_1,\nu_2\nu_3\rbrace $, then equation \ref{S1} becomes
	\begin{equation}\label{S2}
		\frac{d}{dt} \left( |X|^2+|Y|^2 + |Z|^2 \right) \leq -\alpha( |X|^2 + |Y|^2+ |Z|^2 ) - W
	\end{equation}
	Which has a solution
	\begin{equation}
		|X|^2+|Y|^2 + |Z|^2 \leq K e^{-\alpha (t-t_0)}(|X_0|^2+|Y_0|^2 + |Z_0|^2)
	\end{equation}
	Notice that, for $\nu_1$, $\nu_2$, $\nu_3$ to following positive conditions must hold.
	\begin{eqnarray}
		2\mu_1 + (1-\eta)\beta b_1 & > & (1-\eta)\beta \frac{\Lambda_M}{min(\mu_1 , \mu_2)} + q  \label{cond11} \\ 
		2 \mu_2 + q & > & (1-\eta)\beta \left(  b_1 +  \frac{\Lambda_M}{min(\mu_1, \mu_2)} \right)  + (1-\epsilon )p \label{cond2}\\ 
		2\mu_3 & > & (1-\epsilon)p +  \frac{(1-\eta) \beta \Lambda_M}{min(\mu_1 , \mu_2)} \label{cond31}
	\end{eqnarray}
\end{proof}

\begin{theorem}
	Suppose $\Lambda : \mathbb{R} \rightarrow [\Lambda_m\, ,\, \Lambda_M]$, where $0<\Lambda_m<\Lambda-M < \infty $, is continuous, then the system (\ref{NS1} - \ref{NS3} ) has a pullback attractor $\mathcal{A}= \lbrace A(t) : \,\, t\in \mathbb{R} \rbrace$ inside $\mathbb{R}_+^3$.
	Moreover, if $\mu_2> (1-\epsilon)p$, and the conditions (\ref{cond11} - \ref{cond31} ) hold, then the solution of the system is exponentially stable.
\end{theorem}
\begin{proof}
    The proof follows the previous proofs.
\end{proof}

\section{Numerical Results}
To perform numerical simulations, we use numerical solvers to integrate the system of ordinary differential equations (ODEs) over time. In this case, we will use MATLAB to perform the simulations. 

At the disease-free equilibrium point $(\frac{\Lambda}{\mu_1}, 0,0)$, parameters have to satisfy condition \eqref{cond-1}. We will use the parameters from Table \ref{T2} that satisfy this condition and present the results as follows:
\begin{table}[ht]
	\caption{List of parameters that satisfied conditions \ref{cond-11}}\label{T2}
	\centering
	\begin{tabular}{|c||c|c|c|c|c|c|c|c|c|}
		\hline 
		parameters&$\Lambda$ & $\mu_1$     & $\mu_2$ & $\mu_3$ & $\beta$ & $\eta$ & $\epsilon$ & $p$& $q$ \\
		\hline
		values& 9.8135&  2          &3        & 7       &0.2  & 0.2  &  0.5    &0.01  &5 \\
		\hline
	\end{tabular}
\end{table}

\begin{figure}[ht]
\centering
	\scalebox{.40}{\includegraphics{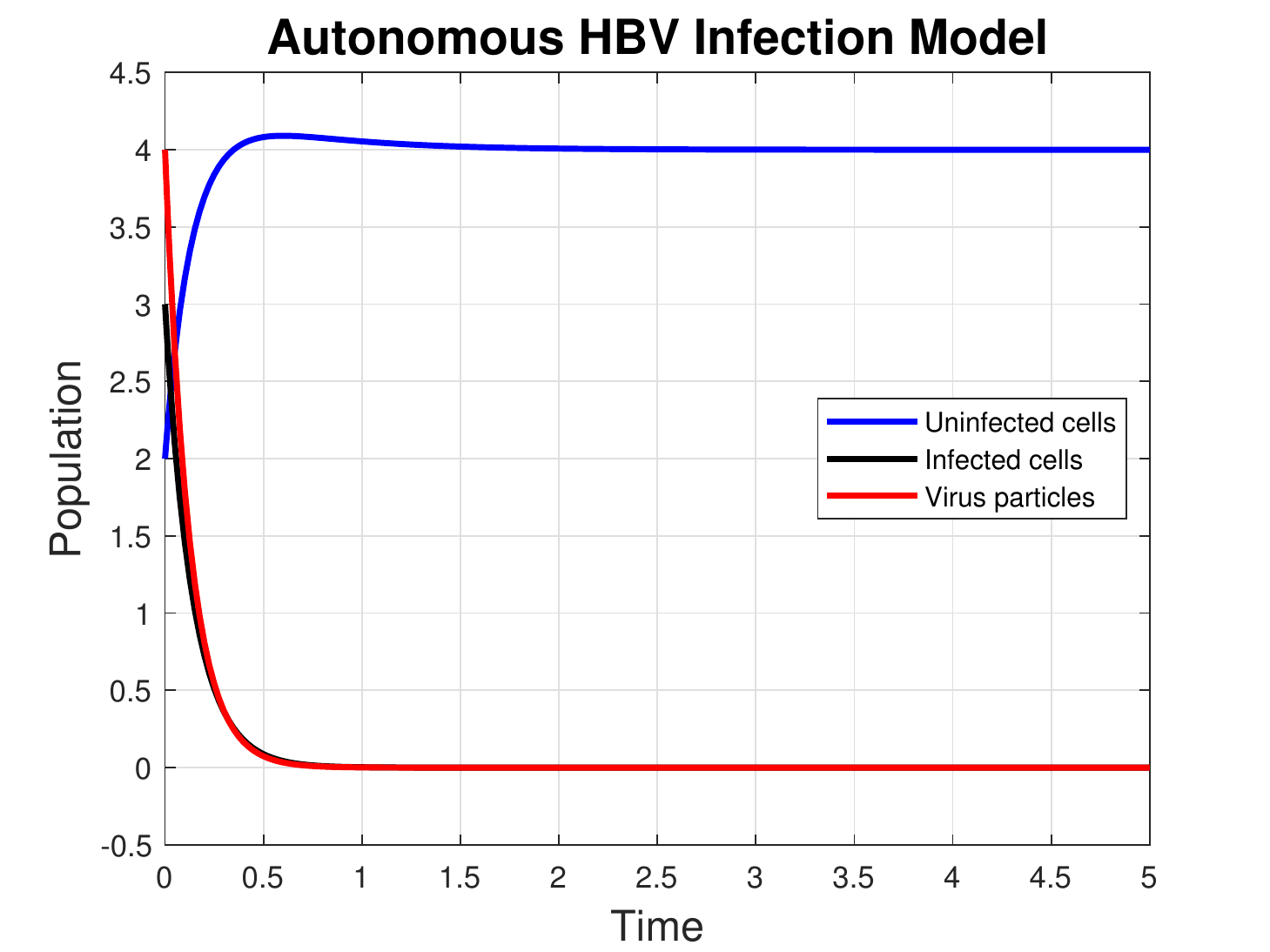}} 
	\caption{The Solution of the model \ref{sys-2} around Diseases-free equilibrium. } 
\end{figure}

\begin{table}[ht]
	\caption{List of parameters that satisfied conditions \ref{condend}}
	\centering
	\begin{tabular}{|c||c|c|c|c|c|c|c|c|c|}
		\hline 
		parameters&$\Lambda$ & $\mu_1$     & $\mu_2$ & $\mu_3$ & $\beta$ & $\eta$ & $\epsilon$ & $p$& $q$ \\
		\hline
		values& 100&  5          &7        & 2       &0.7  & 0.2  &  0.2    &2  &6 \\
		\hline
	\end{tabular}
\end{table}

\begin{figure}[ht]
	\centering
	\scalebox{.3}{\includegraphics{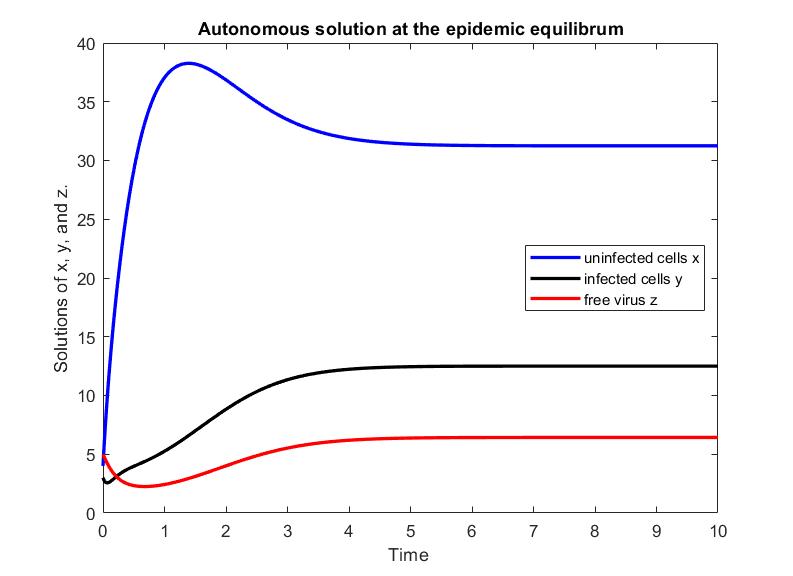}}  
	\caption{Numerical simulation of the autonomous HBV infection model at the epidemic equilibrium.}
\end{figure}


\subsection{Nonautonomous Case}
Figure \ref{nonauto-fig} shows the solutions of the system \ref{nonauto} using an appropriate set of parameters that satisfied the necessary conditions. We approximate the healthy cells' productive function by $\Lambda(t) = cos(2t+\pi/3)+10$, which is a positive and bounded function. On the interval $[0,\,5]$ for the other parameters in the table \ref{set1}.

\begin{table}[ht]
	\caption{ Set of parameters that satisfy the required conditions}\label{set1}
	\centering
	\begin{tabular}{|c|c|c|c|c|c|c|c|c|}
		\hline 
		$\mu_1$     & $\mu_2$ & $\mu_3$ & $\beta$ & $\eta$ & $\epsilon$ & $p$& $q$& $\Lambda$ \\
		\hline
		2          &3        & 7       &0.2  & 0.2  &  0.5    &0.01  &5 &      12\\
		\hline
	\end{tabular}
\end{table}

\begin{figure}[H]
	\centering
	\scalebox{.25}{\includegraphics{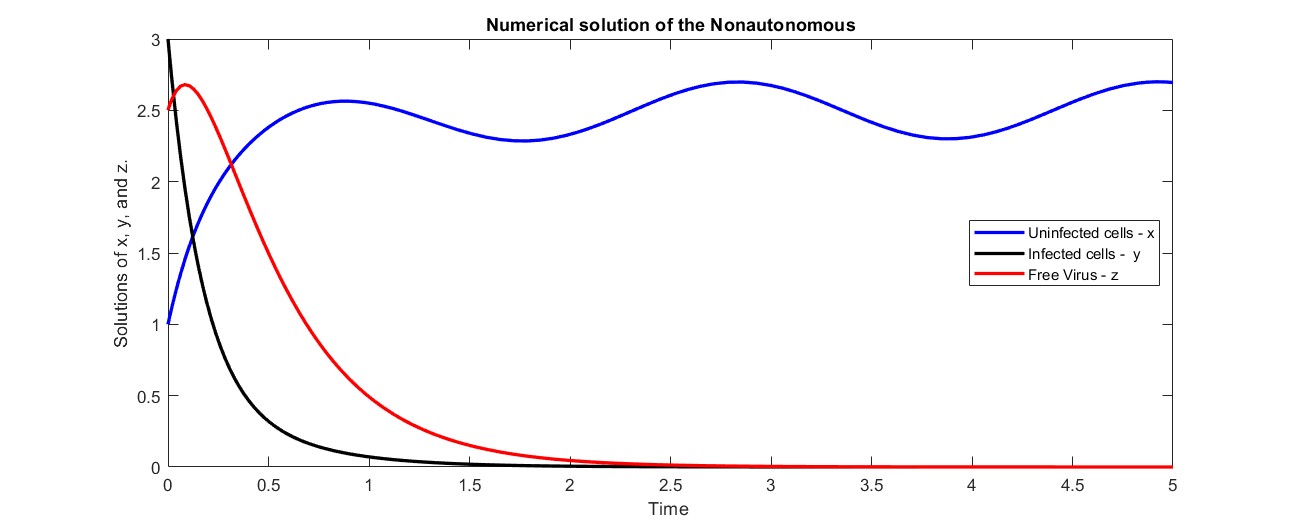}}  
	\caption{Numerical simulations of the nonautonomous HBV infection model at the disease-free equilibrium (DFE). } \label{nonauto-fig}
\end{figure}

\begin{table}[H]
	\caption{ This set of parameters satisfy both  Auto/nonautonomous conditions} \label{set2}
	\centering
	\begin{tabular}{|c|c|c|c|c|c|c|c|c|}
		\hline 
		$\mu_1$     & $\mu_2$ & $\mu_3$ & $\beta$ & $\eta$ & $\epsilon$ & $p$& $q$& $\Lambda$ \\
		\hline
		6         &7        & 0.1       &0.3  & 0.5  &  0.1    &5  &10 &      20\\
		\hline
	\end{tabular}
\end{table}

\begin{figure}[H]
	\centering
	\scalebox{.23}{\includegraphics{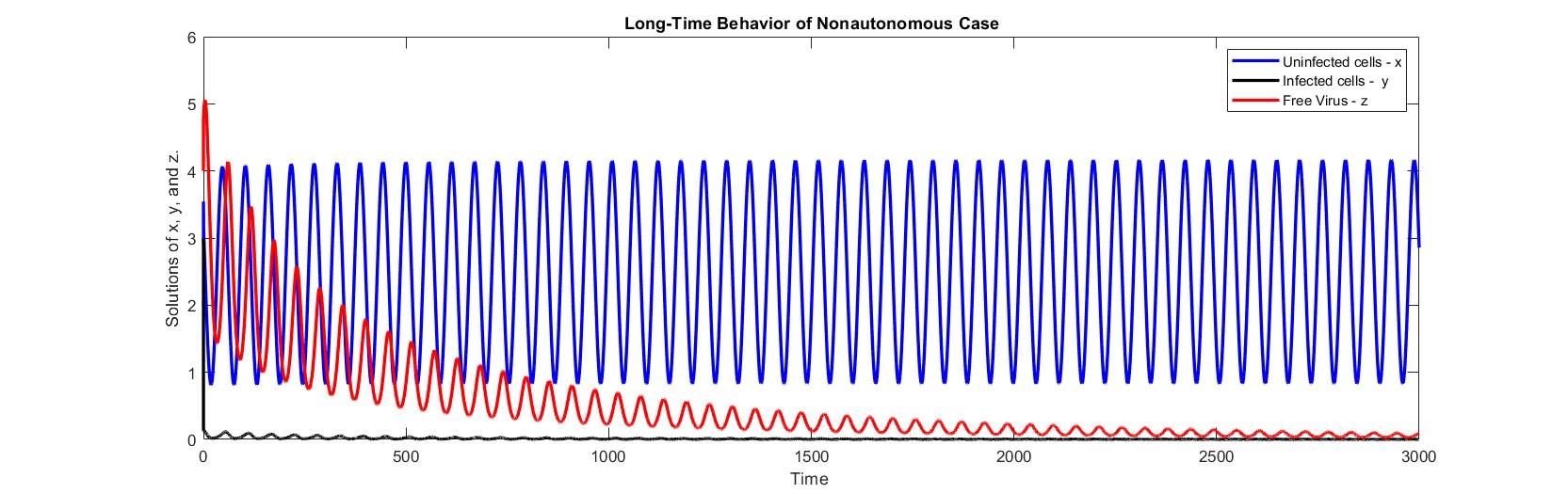}}  
	\caption{Numerical simulations of the nonautonomous HBV infection model (Equations \ref{NS1}, \ref{NS2}, and \ref{NS3}) with time-dependent production number $\Lambda(t)$. }\label{set2-fig1}
\end{figure}
For the same set of parameters \ref{set2}, the autonomous model blowup.
\begin{figure}[ht]
	\centering
	\scalebox{.35}{\includegraphics{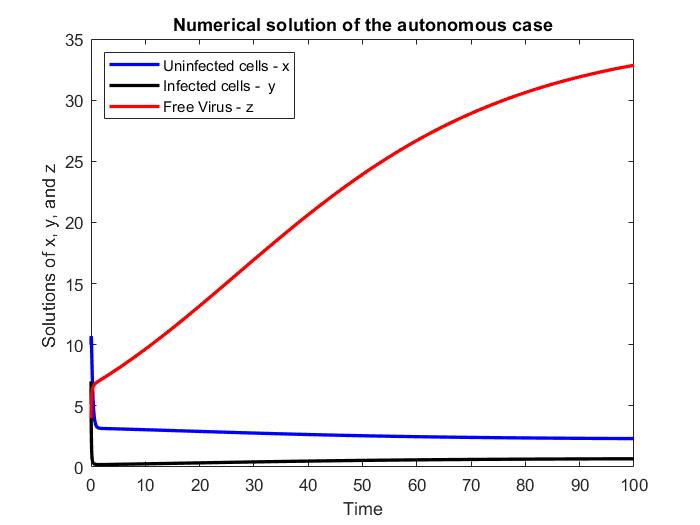}}  
	\caption{The free virus solution $z(t)$  blowup, for the same set of parameters that is used on the nonautonomous case.}\label{set2-fig2}
\end{figure}

\section{Conclusion}
 This study presents an enhanced model for Hepatitis B Virus (HBV) transmission, including autonomous and nonautonomous cases and medical treatment impacts. It validates unique solutions and assesses their positivity over time, with a detailed stability analysis at the equilibrium points. Local and global stability are explored using the Jacobian matrix, $R_0$, and a Lyapunov function, respectively, linking stability conditions to $R_0$. Numerical simulations demonstrate the disease-free equilibrium's stability and provide insights into HBV dynamics and intervention effectiveness. Nonautonomous systems can better represent HBV transmission dynamics by including time-dependent factors, while autonomous systems assume constant parameters. Choosing between the two depends on the research question or application, but nonautonomous models may offer more accurate insights into real-world situations and control strategies.

\bibliographystyle{splncs04}
\bibliography{References}

\end{document}